\documentclass[11pt]{amsart}
\usepackage{stmaryrd}
\usepackage{amssymb}
\usepackage{csvsimple}
\usepackage{siunitx}
\usepackage{hyperref}
\usepackage{url}
\usepackage{graphicx,color}

\usepackage[style=trad-abbrv,url=false,doi=true,
  eprint=true,isbn=false,backend=biber]{biblatex}
\graphicspath{{pics/},{figs/}}

\def\R{\mathbb{R}}

\def\cI{\mathcal{I}}

\def\cK{\mathcal{K}}

\def\cN{\mathcal{N}}
\def\cO{\mathcal{O}}

\def\cS{\mathcal{S}}
\def\cT{\mathcal{T}}

\def\a{\alpha}
\def\b{\beta}

\def\d{\delta}

\def\p{\partial}

\def\vphi{\varphi}
\def\O{\Omega}

\def\transp{{\sf T}}
\def\hu{\widehat{u}}
\def\tu{\widetilde{u}}

\newcommand{\dv}[1]{\,{\mathrm d}#1}

\newcommand{\wcheck}[1]{#1\hspace{-.8ex}\mbox{\huge {\lower.45ex \hbox{$\textstyle \check{}$}}} \hspace{.5ex}}

\let\oldmarginpar\marginpar
\renewcommand\marginpar[1]{
  \oldmarginpar[\raggedleft\footnotesize #1]
  {\raggedright\footnotesize #1}}
\newtheorem{definition}{Definition}
\newtheorem{lemma}[definition]{Lemma}

\newtheorem{theorem}[definition]{Theorem}

\newtheorem{remark}[definition]{Remark}

\numberwithin{definition}{section}
\definecolor{modmag}{RGB}{179,0,229}

\def\tu{\widetilde{u}}
\renewcommand{\hat}{\widehat}

\renewcommand{\d}{\textnormal{d}}

\def\av{{\rm av}}

\begin{document}
\title[Error analysis for harmonic map heat flow]{Error analysis
for the numerical approximation of the harmonic map heat flow with
nodal constraints}
\author{S\"oren Bartels}
\address{Abteilung f\"ur Angewandte Mathematik,  
Albert-Ludwigs-Universit\"at Freiburg, Hermann-Herder-Str.~10, 
79104 Freiburg i.~Br., Germany}
\email{bartels@mathematik.uni-freiburg.de}
\author{Bal\'{a}zs Kov\'{a}cs}
\address{Faculty of Mathematics, Technical University of Munich,
	Boltzmannstr.~3, 85748 Garching bei M\"unchen, Germany, 
	And 
        Faculty of Mathematics,  University of Regensburg, 
	Universit{\"a}tsstr.~31, 93040 Regensburg, Germany}
\email{balazs.kovacs@mathematik.uni-regensburg.de}
\author{Zhangxian Wang}
\address{Abteilung f\"ur Angewandte Mathematik,  
Albert-Ludwigs-Universit\"at Freiburg, Hermann-Herder-Str.~10, 
79104 Freiburg i.~Br., Germany}
\email{zhangxian.wang@mathematik.uni-freiburg.de}
\date{\today}
\renewcommand{\subjclassname}{
\textup{2010} Mathematics Subject Classification}
\begin{abstract}
An error estimate for a canonical discretization of the 
harmonic map heat flow into spheres is derived. The numerical
scheme uses standard finite elements with a nodal treatment
of linearized unit-length constraints. The analysis is based
on elementary approximation results and only uses the discrete
weak formulation. 
\end{abstract}

\maketitle

\section{Introduction}
The harmonic map heat flow into spheres is obtained as the $L^2$ gradient
flow of the Dirichlet energy on vector fields satisfying a pointwise unit 
length condition, i.e., for the functional
\[
E(u) = \frac12 \int_\O |\nabla u|^2 \dv{x}, 
\quad u\in H^1(\O,\R^m), \quad |u|^2= 1.
\]
Given initial data $u^0\in H^1(\O,\R^m)$ with $|u^0(x)|^2 =1$ for almost
every $x\in \O$ the evolution problem reads in strong form
\begin{equation}\label{hhf_strong}
\p_t u - \Delta u = |\nabla u|^2 u, \quad u(0) = u^0, \quad |u|^2 = 1,\quad
\nabla u|_{\p\O} \cdot n = 0.
\end{equation}
The problem admits possibly non-unique weak solutions which satisfy the
energy decay property
\begin{equation}\label{eq:enerlaw}
E(u(t)) + \int_0^t \|\p_t u\|^2 \dv{s} \le E(u^0),
\end{equation}
and the weak formulation of the evolution problem which is given by 
\begin{equation}\label{hhf}
(\p_t u,\phi) + (\nabla u,\nabla \phi) = 0, \quad 
\end{equation}
for all $t\in [0,T]$ with test functions $\phi \in H^1(\O,\R^m)$
satisfying the orthogonality condition
\[
\phi(x) \cdot u(t,x) = 0,
\]
which arises as a linearization of the unit-length constraint. The time
derivative $\p_t u$ satisfies this condition, i.e., we
have 
\[
\p_t u (t,x) \cdot u(t,x) = 0.
\]
The well-posedness of the problem and properties of solutions have been
investigated in, e.g.,~\cite{EellSamp64,ChenStru89}.
For the developement of numerical methods it is attractive to 
exploit the elementary
fact that the orthogonality implies the preservation of the unit-length
constraint, i.e., the identity $\p_t |u|^2 = 2 \p_t u \cdot u$  or equivalently,
\begin{equation}\label{constr_pres}
|u(t,x)|^2 - |u^0(t,x)|^2  = 2 \int_0^t \p_t u(s,x) \cdot u(s,x) \dv{s},
\end{equation}
yields that $|u|^2= 1$ almost everywhere provided that the initial data has this
property and the pointwise orthogonality $\p_t u \cdot u = 0$ is fulfilled. 
An important observation for the derivation of error estimates is
that regular solutions are unique and that a local stability result holds.
Instead of deriving such a result from the full Euler--Lagrange
equations~\eqref{hhf_strong}
we follow the novel approach from~\cite{AFKL2021} and
consider the tangent space formulation 
\[
\p_t u = P(u) \Delta u 
\]
with the tangential projection $P(s) = I -  s s^\transp$ at $s\in \R^m$.
If $u_*$ is an approximate solution 
satisfying $|u_*|^2 =1$ almost everywhere, we
define its defect $d_*$ and residual $r$ via 
\[
\p_t u_* = P(u_*) \Delta u_* + d_* = P(u) \Delta u_* + r,
\]
where $r= d_* - (P(u)-P(u_*))\Delta u_*$. This relation allows us
to compare the equations and obtain an evolution equation for the 
error $e= u-u_*$, i.e., 
\[
\p_t e = P(u) \Delta e -r,
\]
or in weak form
\[
(\p_t e, \phi) + (\nabla e, \nabla \phi) = -(r,\phi)
\]
for all $\phi \in H^1(\O,\R^m)$ with $\phi \cdot u(t,\cdot) = 0$. The
function $\p_te $ is an attractive test function to obtain an error
estimate but may not be admissible. We thus consider its projection and
note that, using, e.g., $P(u)\p_t u = \p_t u$, 
\[
\phi = P(u) \p_t e = P(u) \p_t u - P(u) \p_t u_*  = \p_t e - q
\]
with $q = - (P(u) - P(u_*)) \p_t u_*$.  This implies that
\[
\|\p_t e\|^2 + (\nabla e, \nabla \p_t e) = -(r, \p_t e + q ) 
- (\p_t e,q) - (\nabla e, \nabla q).
\] 
Local Lipschitz estimates for the projection operator, cf. 
formula~\eqref{lemma:lip_p} below, yield that
\[
\|q\|_{H^k} \le c_{a,k} \|e\|_{H^k}, \quad \|r\|\le c_{b,0} \|e\| + \|d_*\|,
\]
where the constants $c_{a,k}$, $k=1,2$, and $c_{b,0}$ depend on higher
order norms of $u_*$. We thus arrive at 
\[
\|\p_t e\|^2 + \frac{\d}{\d t} \|\nabla e\|^2 \le c_1 \|e\|_{H^1}^2 + c_2 \|d_*\|^2.
\]
Using that $\frac{\d}{\d t} \|e\|^2\le \|\p_t e\|^2 + \|e\|^2$
a Gronwall argument proves  
\[
\sup _{t\in[0,T]}\|e(t)\|^2_{H^1} \le  \Big(\|e(0)\|_{H^1}^2 + c_2 \int_0^T \|d_*\|^2 \dv{t} \Big)
\exp\big((c_1+1) T\big).
\]
In a semi- or fully discrete setting the approximations $u_*$ 
are suitable interpolants of a sufficiently regular exact solution 
of~\eqref{hhf} and $u$ is replaced by the solution of 
the numerical scheme. 

As an example we consider the semi-discrete scheme which, following
an idea by Alouges~\cite{Alou97},  determines for
a step size $\tau>0$ the sequence $(u^n)_{n=0,\dots,N} \subset H^1(\O,\R^m)$
via computing for given $u^{n-1}$ the function $d_t u^n \in H^1(\O,\R^m)$ 
satisfying $d_t u^n \cdot u^{n-1} = 0$ in $\O$ and the linear system 
\[
(d_t u^n,\phi) + (\nabla u^n,\nabla \phi) = 0  
\]
for all $\phi\in H^1(\O,\R^m)$ satisfying $\phi \cdot u^{n-1} = 0$ in $\O$. 
The new approximation $u^n$ is given by $u^n= u^{n-1}+\tau d_t u^n$, in particular
$d_t u^n$ is the backward difference quotient. By choosing 
$\phi = d_t u^n$ we find the unconditional energy stability
\[
\frac12 \|\nabla u^{N'}\|^2 + \tau \sum_{n=1}^{N'} \|d_t u^n\|^2 \le \frac12 \|\nabla u^0\|^2.
\]
for all $N'=1,2,\dots, N$. Furthermore, as observed in~\cite{Bart16}
the orthogonalities lead to the discrete version of relation~\eqref{constr_pres} given
by
\[
|u^n|^2 = |u^{n-1}|^2 + \tau^2 |d_t u^{n-1}|^2 = \dots = 1 + \tau^2 \sum_{j=1}^n |d_t u^j|^2,
\]
so that $|u^n|^2 \ge 1$ and $\||u^n|^2-1\|_{L^1} \le (\tau/2) \|\nabla u^0\|^2$, i.e.,
the constraint-violation is of order $O(\tau)$. The pointwise normalization of $u^n$,
given by 
\[
\tu^n_{{\rm nor}} = \frac{u^n}{|u^n|},
\] 
is well defined and energy-decreasing which motivates considering $\tu^n_{{\rm nor}}$ as
the new approximation. 
The energy-decreasing property of the normalization is however in general not 
satisfied in fully discrete settings, cf.~\cite{Bart05},
and therefore omitted. Moreover, including
the projection in the scheme makes the numerical analysis more complicated as,
e.g., $d_t u^n$ is not the backward difference quotient anymore if $u^n$ is replaced
by $\tu^n_{{\rm nor}}$. Nevertheless, our analysis shows that $\tu^n_{{\rm nor}}$ approximates
the exact solution $u(t_n)$ with the same order as $u^n$ which justifies the normalization
as a postprocessing procedure. 

With the auxiliary variable $\hu^u = u^{n-1}/|u^{n-1}|$
the iterates of the semi-discrete scheme satisfy 
\[
d_t u^n = P(\hat{u}^n)\Delta u^n,
\]
which leads to defining the defects $d^n$ of the time-step 
evaluations $u_*^n = u(t_n)$, 
$t_n = n\tau$, via 
\[
d^n = P(\hat{u}_*^n)(d_t u_*^n - \Delta u_*^n).
\]
Since $P(u_*^n)(\p_t u(t_n) - \Delta u_*^n ) = 0$ we find that
\[
d^n =P(\hat{u}_*^n) \big(d_t u_*^n - \p_t u(t_n)\big) + 
\big( P(\hat{u}_*^n) - P(u_*^n)\big)(\p_t u(t_n) - \Delta u_*^n),
\]
and hence, with $I_n = [t_{n-1},t_n]$, for $u$ sufficiently regular,
using a local Lipschitz estimate for $P$, 
\[\begin{split}
\|d^n\| &\le \|d_t u_*^n - \p_t u(t_n)\| + c_u \|\hat{u}_*^n - u_*^n\| \|\p_t u(t_n) - \Delta u(t_n)\|_{L^\infty}\\
&\le \tau \| \p_t^2 u\|_{C^0(I_n,L^2)} + c_u \tau \|\p_t u\|_{C^0(I_n,L^2)} \|\p_t u - \Delta u\|_{C^0(I_n,L^\infty)},
\end{split}\]
and thus 
\[
\tau\sum_{k=1}^n \|d^k\|^2 \le c \tau^{2} \|\p_t^2 u\|^2_{L^2([0,T],L^2)}.
\]
Using the strategy of the continuous perturbation result described above,
we obtain the error estimate
\[
\max_{n=0,\dots,N} \|u(t_n) - u^n\|_{H^1} \le c_{{\rm sd}} \tau,
\]
provided that $u$ is sufficiently regular. 

The fully discrete numerical scheme analyzed in~\cite{AFKL2021} imposes
the orthogonality condition in an averaged sense, i.e., via
$\Pi_h (\hu_h^n \cdot \phi_h) = 0$ with the $L^2$ projection $\Pi_h$ onto
the underlying scalar finite element space. This definition of a discrete
tangent space gives rise to a selfadjoint projection operator with
suitable stability and approximation properties. Practically more efficient,
in particular in view of the development of efficient iterative solvers~\cite{KPPRS19}
and the generalization of the methods to pointwise constraints in mechanical
applications~\cite{Bart15-book,BarRei21}, a nodewise variant of orthogonality appears attractive.
In~\cite{AFKL2021} it is argued that this can be analyzed by considering
a suitable correction term. Here, we aim at a direct numerical analysis
for the nodal variant of the constraint. We avoid the use of mass lumping
which would provide a selfadjoint projection operator but would restrict
the analysis to lowest order methods. 
Our results follow from basic estimates for nodal interpolation with
and $H^1$ projection onto finite element spaces working only within the discrete 
weak formulation. 

The finite element scheme computes iterates $(u_h^n)_{n=0,\dots,N}$ in a finite
element space $V_h$ and imposes the orthogonality conditions in the nodes of
the underlying element. Hence, for given 
$\hu\in C(\overline{\O};\R^m)$ we consider the discrete tangent space
\begin{equation}\label{eq:discrete tangent space}
T_h(\hu) = \big\{ \phi_h \in V_h: \cI_h(\phi_h \cdot \hu) = 0 \big\},
\end{equation}
where $\cI_h$ is the nodal interpolation operator associated with $V_h$. 
Note that $T_h(\hu)$ only depends on the directions of the nodal values of $\hu$. 
Given an initial value $u_h^0 \in V_h$ we compute the sequence $(u_h^n)_{n=0,\dots,N}\in V_h$
by successively computing discrete time derivatives
$d_t u_h^n \in T_h (\hu_h^n)$ such that
\begin{equation}
\label{full_discr}
(d_t u_h^n,\phi_h) + (\nabla u_h^n,\nabla \phi_h) = 0 
\end{equation}
for all $\phi_h\in T_h(\hu_h^n)$, where $\hu_h^n = u_h^{n-1}/|u_h^{n-1}|$. 
The error analysis becomes substantially more involved as, e.g., boundedness
of the iterates away from zero has to be guaranteed. 
Our main result is the following variant of the results from~\cite{AFKL2021}.

\begin{theorem}[Error estimate] \label{thm:error estiamate}
Assume that $\|u_h^0-u^0\|_{H^1} \le c h$ and 
let $(u_h^n)_{n = 0,...,N}$ be the continuous, piecewise linear 
finite element approximations obtained by \eqref{full_discr} 
on a family of regular and quasi-uniform triangulations of $\O$
with mesh-sizes $h>0$. 
Suppose that the exact solution $u$ of the harmonic map heat 
flow \eqref{hhf_strong} satisfies  
\begin{equation}\label{eq:reg_cond}
\begin{split}
u \in & C^2([0,T],H^1(\O)) \cap C^1([0,T],H^2(\O)  \cap W^{1,\infty}(\O)) \\
& \qquad \cap C^0([0,T],W^{2,\infty}(\O)).
\end{split}
\end{equation}
Then, for $h,\tau>0$ sufficiently small we have that 
\[
\max_{n=0,\dots,N} \|u_h^n - u(t_n)\|_{H^1} \leq c_{{\rm fd}} ( \tau + h ),
\]
provided that $\tau\le c_{{\rm m}} h^{1/2}$ with $c_{{\rm m}}>0$ sufficiently small. 
\end{theorem}

The regularity condition can be weakened to
$u \in H^2([0,T],H^1(\O)) \cap H^1([0,T],H^2(\O)\cap W^{1,\infty}) \cap C^0([0,T],W^{2,\infty}(\O))$.
This requirement and~\eqref{eq:reg_cond} can in general only be expected locally or for
initial data with small initial energies. However, the optimal
convergence of the numerical scheme in the case of smooth solutions
underlines its efficiency and thereby complements weak convergence theories
under minimal regularity assumptions. 

Since our error analysis only uses elementary approximation results for
Lagrange finite element methods, it directly extends to higher-order
finite element methods of polynomial degree $r \ge 1$ with the convergence
rate $O(\tau + h^r)$ under suitable regularity conditions. 
Moreover, by combining our consistency and stability bounds with 
the multiplier techniques used in \cite{AFKL2021}, our error estimates extend
to discretizations using linearly implicit backward difference formulae up to 
order $k\le 5$ with the convergence rate $O(\tau^k+h^r)$
under suitable regularity conditions and in fact a weaker step-size condition.
A crucial modification arises in the stability estimate for the error equation 
which is proved using multiplier techniqe based energy estimates, cf., e.g., 
\cite[Appendix]{AFKL2021}. 

Various convergence theories for the numerical approximations obtained 
with variants of the iteration~\eqref{full_discr} have been established
for the harmonic map heat flow and the closely related Landau--Lifshitz--Gilbert
equations. Motivated by understanding the occurrence of singular solutions
as in~\cite{ChDiYe92},
the weak convergence of subsequences to weak solutions of the evolution
problems has been established in \cite{AlougesJaisson2006,BartelsProhl2006,BBFP07,BartelsProhl2007,Alouges2008,BartelsLubichProhl2009,BanasProhlSchaetzle2010,AKST2014,PRS18,MPPR22}.  
Error estimates with specific convergence rates have been proved under
suitable regularity conditions, cf.~\cite{Proh01-book,Cimrak2005,Gao2014,An2016}.
The method and estimates considered here are a variant of the arguments
given in~\cite{AFKL2021} using a nodal treatment of orthogonalities
but also avoiding a projection step. Recently, error estimates for 
schemes that include such a step to guarantee the unit length condition
in the nodes of a finite element space have been derived 
in~\cite{AnGaoSun2021,AnSun2021} and \cite{GuiLiWang2022}. While they also
lead to optimal error estimates they require more restrictive conditions and
the methods may not be convergent in the absence of a regular solution. A byproduct
of our analysis shows that a postprocessing of our numerical solutions leads
to approximations that obey the length constraint in the nodes and are 
quasi-optimal approximations of the exact solution.  

The outline of this article is as follows. We specify notation and 
state some preliminary results in Section~\ref{sec:prelim}. In Section~\ref{sec:fully_discr}
we devise the fully discrete numerical scheme and derive properties of a
projection operator related to the discrete tangent spaces. 
Section~\ref{sec:consistency} provides stability bounds in terms of
consistency terms. These
lead to the main error estimate derived in Section~\ref{sec:error_est}. 

% \clearpage

\section{Preliminaries}\label{sec:prelim}
We collect in this section some elementary facts about the harmonic map heat
flow and numerical concepts for discretizing parabolic partial differential
equations. We use standard notation for function spaces but often omit 
domains and target spaces when this is clear from the context. We abbreviate
by $(\cdot,\cdot)$ and $\|\cdot\|$ the inner product and norm in $L^2(\O,\R^m)$.
Throughout the article a factor~$c$ denotes a constant that may depend on
regularity properties of an exact solution. 

\subsection{Harmonic map heat flow}
For a bounded domain $\O\subset \R^{n_\O}$, with $n_\O=1,2,3$, a time horizon $T>0$, 
and given initial data $u^0\in H^1(\O,\R^m)$ with $|u^0|^2=1$ almost everywhere
in $\O$ we say that  $u \in H^1((0,T),L^2(\O,\R^m) \cap L^\infty((0,T),H^1(\O,\R^m)$
is a weak solution of the harmonic map heat flow if $|u(t,x)|^2=1$  for almost
every $(t,x) \in [0,T]\times \O$, and $u(t,0) = u^0$, and 
\[
(\p_t u,\phi) + (\nabla u,\nabla \phi) = 0
\]
holds for almost every $t\in [0,T]$ and all $\phi\in H^1(\O,\R^m)$ with 
$\phi(x)\cdot u(t,x) =0$ for almost every $x\in \O$. 
Defining the tangential projection operator $P(s):\R^m\to \R^m$ for $s\in  \R^m$ via  
\[
P(s) = I - s s^\transp,
\]
we may state the the strong form~\eqref {hhf_strong}
of the evolution problem as 
\[
\p_t u = P(u) \Delta u, \quad |u|^2 = 1, \quad u(0,\cdot) = u^0, \quad \nabla u|_{\p\O} \cdot n = 0.
\]
For a function $\hu\in H^1(\O,\R^m)$ satisfying $|\hu|^2=1$ 
we define a tangent space $T(\hu)$ relative to the unit sphere as
\[
T(\hu) = \big\{\phi \in H^1(\O,\R^m): \phi \cdot \hu = 0 \big\}.
\]
We note that we formally have $\p_t u \in T(u)$ and that the weak formulation
uses functions $\phi \in T(u)$. We also note the formal energy law~\eqref{eq:enerlaw}
which follows from choosing $\phi = \p_t u$ for regular solutions 
in the weak formulation of the flow; it can be rigorously established 
for suitably constructed weak solutions, cf.~\cite{Stru96-book}.

\subsection{Time discretization}
Given a step size $\tau>0$ we define the time steps $t_n = n \tau$, $n=0,1,\dots,N$, with $N\ge 0$
maximal such that $t_N \le T$. We also define the time intervals $I_n = [t_{n-1},t_n]$, $n=1,2,\dots,N$. 
For a sequence $(a^n)_{n=0,\dots,N}$ we define the backward difference quotient operator $d_t$ via
\[
d_t a^n = \tau^{-1}(a^n-a^{n-1})
\]
for $n=1,2,\dots,N$. If $u\in H^2(0,T;V)$ and $u^n = u(t_n)$ we have that 
\begin{equation}\label{stab_diff_quot}
\|d_t u^n \|_V = \Big\| \tau^{-1} \int_{I_n} \p_t u (s) \dv{s} \Big\|_V \le \|\p_t u \|_{L^2_\av(I_n,V)},
\end{equation}
where we use the abbreviation $\|\phi\|_{L^2_\av(I_n)} = \tau^{-1/2} \|\phi\|_{L^2(I_n)}$. Obviously,
we have $\|\phi\|_{L^2_\av(I_n)} \le \|\phi\|_{C^0(I_n)}$ if $\phi \in C^0(I_n;V)$.
Moreover, we have 
\begin{equation}\label{est_diff_quot}
\| d_t u^n - \p_t u(t_n)\|_V = \Big\| \frac1\tau \int_{I_n} (t_{n-1}-s)\p_t^2 u(s) \dv{s} \Big\|_V   
\le \frac{\tau}{\sqrt{3}} \|\p_t^2 u\|_{L^2_\av(I_n,V)}.
\end{equation}
% We note that for the discrete
% time-integral of the previous estimate we have 
% \[
% \tau \sum_{n=1}^N \|d_t u^n - \p_t u(t_n)\|_V^2 \le \tau^2 \|\p_t^2 u \|_{L^2([0,T],V)}^2,
% \]
% and the right-hand side is bounded by $\tau^2 T \|\p_t^2 u \|_{C^0([0,T],V)}^2$.
 
\subsection{Space discretization}
For a regular and quasi-uniform triangulation $\cT_h$ of the simplicial domain $\O$ 
with mesh-size $h>0$ we denote the lowest order $C^0$-conforming
finite element space of piecewise linear functions by $\cS^1(\cT_h)$
and abbreviate the corresponding vectorial finite element space by
\[
V_h = \cS^1(\cT_h)^m.
\]
We let $\cN_h$ be the set of vertices of elements and denote the nodal interpolation 
operator applied to scalar or vector-valued functions by
\[
\cI_h : C(\overline{\O};\R^\ell) \to \cS^1(\cT_h)^\ell, \quad \cI_h v = \sum_{z\in \cN_h} v(z) \vphi_z,
\]
where $(\vphi_z:z\in \cN_h)$ is the scalar nodal basis for $\cS^1(\cT_h)$. We
let ${\rm D}_h^2$ denote the elementwise defined Hessian and note that we have
for $k=0,1$ 
\[
\|v - \cI_h v\|_{H^k} \le c h^{2-k} \|{\rm D}_h^2 v\|
\]
for $v\in H^1(\O)$ with $v|_K\in H^2(K)$ for all $K\in \cT_h$. 
We make repeated use of inverse estimates, cf.~\cite[Section~4.5]{BreSco08-book}, 
which read for $v_h \in V_h$
\begin{equation}\label{eq:inv_est_der}
\|\nabla v_h\|_{L^p} \le c h^{-1} \|v_h\|_{L^p}
\end{equation}
and, incorporating a Sobelev inequality for $q=2n_\O$, 
\begin{equation}\label{eq:inv_est_p}
\|v_h\|_{L^\infty} \le c h^{-n_\O/q} \|v_h\|_{L^{q}} \le c h^{-1/2} \|v_h\|_{H^1}.
\end{equation} 
For $n_\O =2$ the factor $h^{-1/2}$ can be replaced by 
$1+|\log h|$, cf., e.g.,~\cite{Bart15-book}, if $n_\O=1$ it can be entirely omitted. 
We also make use of a mean-preserving Ritz projection $R_h: H^1(\O) \to V_h$, 
defined by
\begin{equation}
\label{eq:Ritz def}
(\nabla R_h v,\nabla w_h) + (R_hv,1)(w_h,1) = (\nabla v,\nabla w_h) + (v,1)(w_h,1) ,
\end{equation}
for all $w_h \in V_h$. The element $R_h v \in V_h$ is uniquely defined by
the Lax--Milgram lemma and choosing a constant function $w_h$ yields that
$(R_hv,1) = (v,1)$. We thus have that
\[
(\nabla R_h v,\nabla w_h) = (\nabla v,\nabla w_h)
\]
for all $w_h\in V_h$. If $\nabla v \cdot n = 0$ on $\p\O$ we have that
\[
(\nabla R_h v ,\nabla w_h) = - (\Delta v,w _h).
\]
Besides the standard $H^1$  error estimate for $v\in H^2(\O)$  
\begin{equation}\label{ritz_max_standard}
\|v - R_h v\|_{H^1}  \le ch \|v \|_{H^2},
\end{equation}
we have the (generally suboptimal) $L^\infty$ error estimate 
\begin{equation}\label{ritz_max_norm}
\|v - R_h v\|_{L^\infty}  \le ch \|v \|_{W^{2,\infty}}
\end{equation}
for all $v\in W^{2,\infty}(\O)$ and that $R_h$ is $W^{1,\infty}$ stable, 
cf.~\cite[Section~8.1]{BreSco08-book}. 

\subsection{Normalization estimates}
It will be necessary to normalize vector fields that are uniformly bounded
away from zero, i.e., for $u\in H^1(\O,\R^m)$ with $|u| \ge c_\ell >0$ we define
$N(u)\in H^1(\O,\R^m)$ via 
\[
N(u) = \hu = \frac{u}{|u|}.
\]
Our first estimates concern stability properties of $N$.

\begin{lemma}[Normalization bounds]\label{norm_ests} 
Let $u\in W^{1,\infty}(\O)$ and $u_h\in V_h$ with $0< c_\ell \le |u|,|u_h| \le c_\ell^{-1}$.
We then have that  
\[
\|\nabla N(u)\|_{L^\infty} \le c \| \nabla u\|_{L^\infty}, \quad 
\|{\rm D}_h^2 N(u_h)\|_{L^\infty} \le c \|\nabla u_h \|_{L^\infty}^2. 
\]
\end{lemma}

\begin{proof}
The first estimate follows from the bound
\[ 
 \Big|\p_j \Big(\frac{u}{|u|}\Big)\Big| \le \Big| \frac{\p_j u}{|u|} \Big| 
+ \Big| \frac{u (\p_j u\cdot u)}{|u|^3}\Big|.
\]
Noting that $ \p_i \p_j u_h= 0$ on every $K\in \cT_h$ we verify that
\[ 
\begin{split}
\Big|\p_i \p_j  \Big(\frac{u_h}{|u_h|}\Big)\Big| \le \Big| & \frac{\p_i u_h (\p_j u_h\cdot u_h)}{|u_h|^3}\Big| 
+ \Big|\frac{\p_j u_h (\p_i u_h\cdot u_h) + u_h (\p_j u_h\cdot \p_i u_h)}{|u_h|^3} \Big| \\
& + \Big|\frac{u_h (u_h\cdot \p_i u_h)(u_h\cdot \p_j u_h)}{|u_h|^5} \Big|,
\end{split}
\]
and deduce the second bound. 
\end{proof}

The operator $N$ is locally Lipschitz continuous.

\begin{lemma}[Local Lipschitz estimate]\label{la:normalize_lip}
Let $k\in \{0,1\}$ and $1\le p \le \infty$. For all
$u,\tu\in W^{k,p}(\O,\R^m)$ with $0<c_\ell \le |u|,|\tu|\le c_\ell^{-1}$ in $\O$ 
we have
\[
\|N(u)-N(\tu)\|_{W^{k,p}} \le c \|u-\tu\|_{W^{k,p}}.
\]
\end{lemma}

\begin{proof}
The estimate for $k=0$ follows from the inequality
\[ % \label{ineq1}
\left|\frac{u}{|u|}-\frac{\tu}{|\tu|}\right| = \left|\frac{u(|\tu|-|u|) + |u|(u-\tu)}{|u||\tu|}\right| 
\le 2 \min\{|u|^{-1},|\tu|^{-1}\}|u-\tu|.
\]
If $k=1$ we use $\p_i N(u) = \p_i u/|u| - u (\p_i u \cdot u)/|u|^3$ 
% \[
% \Big|\p_i \big(N(u) -N(\tu)\big)\Big|
% \le \Big|\frac{\p_i u}{|u|} - \frac{\p_i \tu}{|\tu|}  \Big|
%+ \Big|\frac{u (u\cdot \p_i u)}{|u|^3} - \frac{\tu (\tu \cdot \p_i \tu)}{|\tu|^3}\Big|
% \]
to deduce the bound. 
\end{proof}

Stability properties of the nodal interpolation of normalized vector fields 
are provided by the following lemma. 

\begin{lemma}[Stability of $\cI_h$ on rational expressions]\label{la:interpol_rational}
Given elementwise polynomial functions $q_h,r_h \in C(\overline{\O})$, i.e., 
$q_h|_K,r_h|_K \in P_r(K)^m$ for all $K\in \cT_h$, and
such that $0<c_\ell\le |q_h| \le c_\ell^{-1}$ we have for $k\in \{0,1\}$ and $1\le p\le \infty$
that
\[
\Big\|\cI_h \Big(\frac{r_h}{|q_h|}\Big)\Big\|_{W^{k,p}} \le c \Big\|\frac{r_h}{|q_h|}\Big\|_{W^{k,p}}.
\]
\end{lemma}

\begin{proof}
We note that for $K\in \cT_h$ the set of pairs 
\[
\cK = \big\{ (r_h,q_h) \in (P_r(K)^m)^2: \|r_h\|_{L^p(K)} = 1, \, c_\ell \le |q_h|\le c_\ell^{-1} \}
\]
is compact so that the continuous function 
\[
F:\cK \to \R, \quad (r_h,q_h) \mapsto \frac{\|\cI_h (r_h/|q_h|)\|_{L^p(K)}}{{\|r_h/|q_h|\|_{L^p(K)}}}
\]
attains its maximum which implies the $L^p$ stability result. 
With this, we also have the stability in $W^{1,p}$ norms, as, e.g.,
using the inverse estimate~\eqref{eq:inv_est_der}, 
\[
\|\nabla \cI_h (r_h/|q_h|) \|_{L^p(K)} 
\le c h^{-1} \|r_h/|q_h|-\a_K \|_{L^p(K)} \le c \|\nabla(r_h/|q_h|) \|_{L^p(K)},
\]
where $\a_K$ is the mean of $r_h/|q_h|$ on $K$. 
\end{proof}

% \clearpage 

\section{Fully discrete scheme}\label{sec:fully_discr}
In this section we devise the fully discrete time-stepping scheme and state
some elementary properties about the discrete projection operator. 

\subsection{Finite element discretization}
The orthogonality condition included in the time-stepping scheme needs to be suitably
discretized in a fully discrete scheme. Following~\cite{Alou97,Bart05} we impose it
at the nodes of the triangulation and define for $\hu\in C(\overline{\O},\R^m)$, 
with $|\hu|^2 = 1$, a discrete tangent space via 
\[
T_h(\hu) = \big\{ \phi_h\in V_h : \cI_h (\phi_h \cdot \hu) = 0 \big\}.
\]
The scheme~\eqref{full_discr} thus computes the sequence $(u_h^n)_{n=0,\dots,N}$
by determining $d_t u_h^n \in T_h(\hu_h^n)$ with $\hu_h^n = N(u_h^{n-1})$ that fulfills 
\[
(d_t u_h^n,\phi_h) + (\nabla u_h^n,\nabla \phi_h) = 0 
\]
for all $\phi_h\in T_h(\hu_h^n)$ with $u_h^n = u_h^{n-1} + \tau d_t u_h^n$.

We note that the scheme is unconditionally well defined and stable in
the sense that solutions satisfy energy estimates, e.g., choosing 
$\phi_h = d_t u_h^n$ and using the binomial formula
\[
(\nabla u_h^n,\nabla d_t u_h^n) = 
\frac{d_t}{2} \|\nabla u_h^n\|^2 + \frac\tau 2 \|\nabla d_t u_h^n\|^2,
\]
we deduce that for $N'=1,2,\dots,N$ we have 
\[
\frac12 \|\nabla u_h^{N'}\|^2 + \tau \sum_{n=1}^{N'} \|d_t u_h^n\|^2 
\le \frac12 \|\nabla u_h^0\|^2.
\]
Moreover, we have the controlled violation of the constraint, i.e., arguing as in the derivation
of~\eqref{constr_pres} one finds that
\begin{equation}\label{constr_pres_discr}
|u_h^n(z)|^2 - |u_h^0(z)|^2 = \tau^2 \sum_{j=1}^n |d_t u_h^j(z)|^2.
\end{equation}
Noting that the iteration satisfies an energy decay property the term
on the right-hand side is of order $\cO(\tau)$ after summation over the
nodes $z\in \cN_h$. 

\subsection{Discrete projection}
To quantify consistency properties of the fully discrete method
we will often make use of the operator $P_h$ defined for $\hu\in C(\overline{\O};\R^m)$  
and $v_h\in V_h$ by
\[
P_h := \cI_h P, \quad v_h \mapsto \cI_h (P(\hu) v_h).
\]
Although it will not be used below, we note that it defines
a projection onto $T_h(\hu)$ with respect to the inner product
\[
(v,w)_h = \int_\O \cI_h (v\cdot w) \dv{x} = \sum_{z\in \cN_h} \b_z v(z)\cdot w(z) 
\]
for $v,w\in C(\overline{\O},\R^m)$ with $\b_z = (1,\vphi_z)$ for all $z\in \cN_h$. 
Note that this is in general not an inner product for higher-order methods. 

\begin{lemma}[Discrete projection]\label{lemma:projection properties}
Let $\hu\in C(\overline{\O})$ with $|\hu|^2 = 1$. 
The operator $P_h = \cI_hP$ is self-adjoint with respect to $(\cdot,\cdot)_h$
i.e.,~for any $v_h,w_h \in V_h$ we have 
\[
(P_h(\hu)v_h,w_h)_h = (v_h,P_h(\hu)w_h)_h.
\]
Moreover, we have $P_h(\hu) v_h \in T_h(\hu)$ for every $v_h\in V_h$ and
\[
(v_h - P_h(\hu)v_h,w_h)_h = 0
\]
for all $w_h\in T_h(\hu)$, i.e., $P_h$ is an orthogonal projection onto $T_h(\hu)$
with respect to $(\cdot,\cdot)_h$.  
\end{lemma}

\begin{proof}
For every $z\in \cN_h$ we have that  $P(\hu(z))= I - \hu(z)\hu(z)^\transp$ is symmetric 
and hence
\[
(P(\hu)v_h \cdot w_h)(z) = (v_h\cdot P(\hu)w_h)(z),
\]
so that a summation over $z\in \cN_h$ yields the self-adjointness. Moreover, we
find that $(P(\hu) v_h)(z) \cdot \hu(z) = 0$ so that $P_h(\hu)v_h \in T_h(\hu)$.
This and the self-adjointness imply the asserted orthogonality relation. 
\end{proof}

\begin{remark}
The linearization of the length constraint at the nodes follows the 
aproaches from \cite{AlougesJaisson2006,BartelsProhl2007}. An averaged
version of the related orthogonality has been considered in~\cite{AFKL2021}
by defining 
\[
T_h^{{\rm avg}}(u) =\big\{ \phi_h\in V_h : \Pi_h(u \cdot \phi_h) = 0 \big\},
\]
with the (scalar) $L^2$ projection $\Pi_h$ onto a finite element space. 
Defining $P_h^{{\rm avg}}$ as the $L^2$ projection onto $T_h^{{\rm avg}}$ 
leads to various stability estimates that require subtle arguments. 
The nodal variant considered here leads to simpler proofs of the 
estimates and allows for a straightforward numerical realization.
\end{remark}

\subsection{Further properties of $P_h$}
We next establish stability and approximation properties of the discrete 
projection operator $P_h = \cI_h P$. Similar properties were shown 
in \cite[Section~5]{AFKL2021} for the averaged discrete tangential projection 
$P_h^{{\rm avg}}$. Although the results are similar, the proofs given here
are immediate consequences of nodal interpolation estimates. 

\begin{lemma}[Approximation]\label{lemma:approx_ph}
For $u_h \in V_h$ with $1/2 \le |u_h|\le 2$ define $\hu_h = N(u_h)$.
For $k\in\{0,1\}$ and $v_h\in V_h$ we have
\[
\|(P_h(\hu_h)-P(\hu_h))v_h\|_{H^k} \le c h^{2-k} \|v_h\|_{W^{1,p}}  \|\nabla u_h\|_{L^\infty}^2.
\]
\end{lemma}

\begin{proof} 
We deduce the estimates from corresponding elementwise estimates. Since 
$v_h|_K$ is linear for every $K\in \cT_h$ we have 
\[
\begin{split}
\|(P_h(\hu_h) - P(\hu_h))v_h\|_{H^k} &\le ch^{2-k} \|{\rm D}_h^2(P(\hu_h)v_h)\| \\
&\le ch^{2-k} \|v_h\|_{H^1}(\|{\rm D}_h^2 \hu_h\|_{L^\infty} + \|\nabla \hu_h\|_{L^\infty}^2).
\end{split}
\]
Incorporating Lemma~\ref{norm_ests}  yields the estimate.
\end{proof}

For the continuous projection operator we have the local Lipschitz estimates
from~\cite[Lemma~4.1]{AFKL2021}, i.e., for $k\in \{0,1\}$ and $u,\tu\in W^{k,\infty}(\O)$ 
with $|u|,|\tu| \le 1$ in $\O$ and all $v\in W^{1,\infty}(\O)$ we have
\begin{equation}\label{lemma:lip_p}
\|(P(u) - P(\tu))v\|_{H^k} \le c \|v\|_{W^{k,\infty}} \|\tu\|_{W^{k,\infty}} \|u-\tu\|_{H^k}.
\end{equation}
The estimate follows from the identity 
\begin{equation}\label{proj_cont_diff}
-(P(u)-P(\tu)) = ee^\transp + e \tu^\transp  + \tu e^\transp 
\end{equation}
with $e=u-\tu$. The discrete projection $P_h$ satisfies similar estimates. 

\begin{lemma}[Discrete local Lipschitz estimate]\label{lemma:lip_ph}
Let $u_{*,h},u_h \in V_h$ such that  $1/2\le |u_{*,h}|, |u_h| \le 2$ and define 
$\hu_{*,h}= N(u_{*,h})$ and $\hu_h = N(u_h)$. Then, for all 
$v_h\in V_h$ and $k\in \{0,1\}$ we have that 
\[
\|(P_h(\hu_{*,h}) - P_h(\hu_h))v_h\|_{H^k} \le  
c \|v_h\|_{W^{k,\infty}} \|u_{*,h}\|_{W^{k,\infty}} \|u_{*,h}- u_h\|_{H^k}.
\]
and
\[
\|(P_h(\hu_{*,h}) - P_h(\hu_h))v_h\|_{L^1} \le c \|v_h\| \|u_{*,h}- u_h\|.
\]
\end{lemma}

\begin{proof}
Noting that, e.g., $P_h (\tu_{*,h})v_h = \cI_h (P(\cI_h \tu_{*,h})v_h)$,
we deduce from Lemma~\ref{la:interpol_rational} that 
\[\begin{split}
\|(P_h(\hu_{*,h}) - P_h(\hu_h))v_h\|_{W^{k,p}} 
&=  \|\cI_h[(P(\cI_h \hu_{*,h}) - P(\cI_h \hu_h))v_h]\|_{W^{k,p}} \\
&\le c \|(P(\cI_h \hu_{*,h}) - P(\cI_h \hu_h))v_h\|_{W^{k,p}}.
\end{split}\]
With~\eqref{lemma:lip_p} we thus find that  
\[
\|(P_h(\hu_{*,h}) - P_h(\hu_h))v_h\|_{H^k} 
\le c \|v_h \|_{W^{k,\infty}} \|\cI_h \hu_{*,h} \|_{W^{k,\infty}} \|\cI_h (\hu_{*,h}-\hu_h)\|_{H^k}.
\]
Using Lemmas~\ref{la:interpol_rational} and~\ref{la:normalize_lip}
we verify the first estimate. The estimate in $L^1$ is obtained similarly
using a H\"older inequality, the bound~\eqref{proj_cont_diff}, and
Lemmas~\ref{la:interpol_rational} and~\ref{la:normalize_lip}.
\end{proof}

Our third result is the $W^{1,p}$ stability of $P_h$. 

\begin{lemma}[Stability]\label{lemma:stab_ph}
For $u_h \in V_h$ with $1/2\le |u_h|\le 2$ let $\hu_h = N(u_h)$.
Then for $p \in \{2,\infty\}$ we have for every $v_h\in V_h$ 
\[
\|P_h(\hu_h) v_h\|_{W^{1,p}} \le c \|v_h\|_{W^{1,p}} \|u_h\|_{W^{1,\infty}}^2.
\]
\end{lemma}

\begin{proof}
We note that $P_h (\tu_h)v_h = \cI_h (P(\cI_h \tu_h)v_h)$ and 
Lemma~\ref{la:interpol_rational}  verify that
\[
\|P_h(\hu_h) v_h\|_{W^{1,p}} 
\le c \|P(\cI_h \hu_h) v_h\|_{W^{1,p}} \le c \| \cI_h \hu_h\|_{W^{1,\infty}}^2 \|v_h\|_{W^{1,p}}.
\]
The application of Lemmas~\ref{la:interpol_rational} and~\ref{la:normalize_lip}
proves the result. 
\end{proof}

% \cleardoublepage

\section{Consistency estimates}\label{sec:consistency}
We derive in this section consistency estimates for the fully
discrete scheme under suitable regularity conditions adapting 
the approach from~\cite{AFKL2021}.

\subsection{Consistency bound}
Letting $u$ be a regular solution of~\eqref{hhf_strong} and using the mean-preserving
Ritz projection $R_h$ and the normalization operator $N$ we define for $n=0,1,\dots,N$ 
\[ 
u_*^n = u(t_n), \quad u_{*,h}^n = R_h u_*^n,\quad \hat{u}_{*,h}^n =N(u_{*,h}^{n-1}).
\]
For $u\in C^0(0,T,W^{2,\infty}(\O))$, using~\eqref{ritz_max_norm}, we estimate 
\[
\| |u_{*,h}^n|-1\|_{L^\infty} \le \| u_{*,h}^n - u_*^n\|_{L^\infty} \le ch \|u\|_{C^0(I_n,W^{2,\infty})}.
\]
With this we deduce the uniform upper and lower bounds  
\begin{equation} \label{eq:upper and lower bound on Ritz}
1/2 \le |u_{*,h}^n(x)| \le 2,
\end{equation}
for $n=0,1,\dots,N$ with $h$ sufficiently small. This implies that $\hat{u}_{*,h}^n$ is well 
defined. Both upper and lower bounds will be used repeatedly.
We define the full discretization defect $d_h^n \in T_h(\hat{u}_{*,h}^n)$ via
\begin{equation}\label{def_defect}
(d_h^n,\phi_h)  = (d_t u_{*,h}^n,\phi_h) + (\nabla u_{*,h}^n,\nabla\phi_h)
\end{equation}
for all $\phi_h\in T_h(\hat{u}_{*,h}^n)$.

\begin{lemma}[Consistency] \label{consistency}
Assume that the solution $u$ of~\eqref{hhf_strong} satisfies~\eqref{eq:reg_cond}.
We then have 
\[ 
\|d_h^n\| \le c(h+\tau).
\]
\end{lemma}

\begin{proof}
Letting $D_h^n = d_t u_{*,h}^n - \Delta u_*^n$ the definition of $u_{*,h}^n$ shows
\[
(d_h^n,\phi_h) = (D_h^n,\phi_h)
\]
for all $\phi_h \in T_h(\hat{u}_{*,h}^n)$. We abbreviate
$D^n = \p_t u(t_n) - \Delta u(t_n)$ and use that 
$P(u_*^n) D^n  = 0$ with the symmetric matrix $P(u_*^n)$ to infer that 
\[\begin{split}
(d_h^n,\phi_h) &= (D_h^n ,(P_h(\hu_{*,h}^n) - P(\hu_{*,h}^n)) \phi_h) \\
& \quad + (D_h^n - D^n,P(\hu_{*,h}^n) \phi_h)
+ (D^n ,(P(\hu_{*,h}^n) - P(u_*^n)) \phi_h).
\end{split}\]
Choosing $\phi_h = d_h^n$ leads to 
\[\begin{split}
\|d_h^n\|^2 & \le \|D_h^n\| \|(P_h(\hu_{*,h}^n) - P(\hu_{*,h}^n)) d_h^n\| \\
& \quad + \|D_h^n - D^n \| \|d_h^n\|
+ \|D^n\|_{L^\infty} \|P(\hu_{*,h}^n) - P(u_*^n)\| \|d_h^n\|.
\end{split}\]
Using~\eqref{stab_diff_quot} and $W^{1,\infty}$ stability for $R_h$, we have
\begin{equation}\label{est_dhn}
\| D_h^n\|_{L^\infty}  \le c(\|\p_t u\|_{C^0(I_n,W^{1,\infty})}  + \|\Delta u\|_{C^0(I_n,L^\infty)}). 
\end{equation}
With Lemma~\ref{lemma:approx_ph} and the inverse estimate~\eqref{eq:inv_est_der},
we find that 
\[
\|(P_h(\hu_{*,h}^n) - P(\hu_{*,h}^n) d_h\|
\le c  h \|d_h^n\|  \|\nabla u_*^n\|_{L^\infty}^2  \le c h \|d_h^n\| \|u\|_{C^0(I_n,W^{1,\infty})}^2.
\]
We next note that~\eqref{est_diff_quot} implies
\[
\|D_h^n - D^n\| = \|d_t u_{*,h}^n - \p_t u (t_n) \| \le \tau \|\p_t^2 u\|_{C^0(I_n,L^2)}.
\]
Furthermore, we have
\[
\|D^n\|_{L^\infty} \le \|\p_t u\|_{C^0(I_n,L^\infty)} + \|u\|_{C^0(I_n,W^{2,\infty})}.
\]
Finally, we use that $|u_*^{n-1}|=1$ to deduce with~\eqref{proj_cont_diff}
that
\[\begin{split}
\|P(\hu_{*,h}^n) - P(u_*^n)\| &\le c \big(\|N(u_{*,h}^{n-1}) - N(u_*^{n-1})\| + \|u_*^{n-1} - u_*^n\|\big) \\
&\le c \big(\|u_{*,h}^{n-1} - u_*^{n-1}\| + \|u_*^{n-1} - u_*^n\|\big) \\
&\le c  \big( h \|u\|_{C^0(I_n,H^2)} + \tau \|\p_t u\|_{C^0(I_n,L^2)}\big).
\end{split}\]
A combination of the estimates proves the result. 
\end{proof} 

\subsection{Residual estimate}
The residual measures the violation of the numerical scheme by the Ritz projections
of the exact solution relative to the orthogonality constraint defined by
the numerical solution, i.e., we define $r_h^n \in T_h(\hu_h^n)$ via
\begin{equation}\label{residual}
(r_h^n,\phi_h) = (d_t u_{*,h}^n,\phi_h) + (\nabla u_{*,h}^n,\nabla \phi_h)
\end{equation}
for all $\phi_h\in T_h(\hu_h^n)$. Note that the defect $d_h^n$
defined in~\eqref{def_defect} belongs to
the space $T_h(\hu_{*,h}^n)$. The following lemma controls the difference.

\begin{lemma}[Residual]\label{res_estimate}
Assume that the solution $u$ of~\eqref{hhf_strong} satisfies~\eqref{eq:reg_cond}
and that $1/2 \le |u_h^{n-1}|\le 2$.  We then have that 
\[
\|r_h^n\| \le c \big(\|d_h^n\| + \|u_{*,h}^{n-1} - u_h^{n-1}\|\big).
\]
\end{lemma}

\begin{proof}
Using the definition of $u_{*,h}^n$, noting $P_h(\hu_h^n) \phi_h = \phi_h$,
incorporating the definition of $d_h^n$, and abbreviating
$D_h^n = d_t u_{*,h}^n- \Delta u_*^n$, we have 
\[\begin{split}
(r_h^n,\phi_h) &= (D_h^n, P_h(u_{*,h}^n) \phi_h) +(D_h^n, (P_h(\hu_h^n) - P_h(\hu_{*,h}^n)) \phi_h) \\
&= (d_h^n, P_h(u_{*,h}^n) \phi_h) + (D_h^n,  (P_h(\hu_h^n) - P_h(\hu_{*,h}^n)) \phi_h).
\end{split}\]
Hence, with $\phi_h = r_h^n$ we deduce with Lemmas~\ref{lemma:lip_ph} and~\ref{lemma:stab_ph} 
that
\[
\|r_h^n\|^2 \le  c \big( \|d_h^n\|  
+  \|D_h^n\|_{L^\infty} \|u_h^{n-1} - u_{*,h}^{n-1}\|\big) \|r_h^n\| .
\]
Incorporating~\eqref{est_dhn} implies the result.  
\end{proof}

\subsection{Test function correction}
By subtracting~\eqref{residual} from~\eqref{full_discr}, 
the residuals $r_h^n$ gives rise to the error equation 
\begin{equation}\label{error}
(d_t e_h^n,\phi_h) + (\nabla e_h^n, \nabla \phi_h) = -(r_h,\phi_h)
\end{equation}
with $e_h^n = u_h^n - u_{*,h}^n$ and for all $\phi_h \in T_h(\hu_h^n)$. 
Since $d_t e_h^n$ is in general not an admissible test function, we 
follow~\cite{AFKL2021} and use 
$P_h(\hu_h^n) d_t e_h^n$. The following lemma controls the corresponding correction error. 

\begin{lemma}[Projected test function]\label{test_fn_corr}
Assume that the solution $u$ of~\eqref{hhf_strong} satisfies~\eqref{eq:reg_cond}
and that $1/2 \le |u_h^{n-1}|\le 2$. There exist functions $s_h^n,q_h^n\in V_h$ such that 
\[
(I- P_h(\hu_h^n)) d_t e_h^n = s_h^n + q_h^n 
\]
and, for $k=0,1$, 
\[
\|s_h^n\|_{H^1} \le c(h+\tau), \quad \|q_h^n\|_{H^k} \le c \|e_h^{n-1}\|_{H^k}.
\]
\end{lemma}

\begin{proof}
Since $(I-P_h(\hu_h^n)) d_t u_h^n =0$  we have that
\[\begin{split}
(I-P_h(\hu_h^n)) d_t e_h^n 
&= -(I- P_h(\hu_{*,h}^n)) d_t u_{*,h}^n - (P_h(\hu_{*,h}^n)- P_h(\hu_h^n)) d_t u_{*,h}^n \\
&=: s_h^n + q_h^n. 
\end{split}\]
(i) To estimate $s_h^n$ we note that $(I-P(u_*^n))\p_t u(t_n) = 0$ and write
\[
s_h^n = (\p_t u(t_n)- d_t u_{*,h}^n ) + (P(u_*^n) \p_t u(t_n) - P_h(\hat{u}_{*,h}^n) d_t u_{*,h}^n ) =: \a+\b.
\]
We have
\[\begin{split}
\|\a\|_{H^1} &\le \| d_t u_{*,h}^n - d_t u_*^n\|_{H^1} + \|d_t u_*^n - \p_t u(t_n)\|_{H^1}\\
	&\le ch\| \p_t u\|_{C^0(I_n,H^2)} + c\tau \|\p_t^2 u\|_{C^0(I_n,H^1)}.
\end{split}\]
To estimate for $\b$ we write
\[\begin{split}
\b & = (P_h(\hat{u}_{*,h}^n) - P(\hat{u}_{*,h}^n)) d_t u_{*,h}^n 
+ P(\hat{u}_{*,h}^n)(d_t u_{*,h}^n - \p_t u(t_n)) \\
& \qquad  + (P(\hat{u}_{*,h}^n) - P(u_*^n)) \p_t u(t_n) =: \b_1 +\b_2 + \b_3.
\end{split}\]
Using Lemma \ref{lemma:approx_ph}
and the $H^1$- and $W^{1,\infty}$ stability of $R_h$ leads to 
\[\begin{split}
\|\b_1\|_{H^1} &\le ch \|d_t u_*^n\|_{H^1}\|\nabla u_{*,h}^{n-1}\|_{L^\infty}^2 \\
&\le ch \|\p_t u\|_{C^0(I_n,H^1)} \|u\|_{C^0(I_n,W^{1,\infty})}^2.
\end{split}\]
Using an $H^1$-bound for $P$, Lemma~\ref{norm_ests},
and $H^1$-stability of $R_h$ shows that 
\[\begin{split}
\|\b_2\|_{H^1} & \le c\| \hat{u}_{*,h}^n\|_{W^{1,\infty}}^2 \|d_t u_{*,h}^n - \p_t u(t_n)\|_{H^1}\\
&\le c \|u_{*,h}^{n-1}\|_{W^{1,\infty}}^2 (\|d_t u_{*,h}^n - d_t u_*^n\|_{H^1} + \|d_t u_*^n - \p_t u(t_n)\|_{H^1})\\
&\le c \|u\|_{C^0(I_n;W^{1,\infty})}^2 (h \|\p_t u\|_{C^0(I_n,H^2)} + \tau \|\p_t^2 u\|_{C^0(I_n,H^1)}),
\end{split}\]
Finally, using~\eqref{lemma:lip_p} and Lemma~\ref{la:normalize_lip}
we verify that 
\[\begin{split}
\|\b_3\|_{H^1} &\le c \|u_*^n\|_{W^{1,\infty}} \|\p_t u(t_n)\|_{W^{1,\infty}} \|\hat{u}_{*,h}^n - u_*^n\|_{H^1}\\
&\le c(\|u_{*,h}^{n-1} - u_*^{n-1}\|_{H^1} + \| u_*^{n-1} - u_*^n\|_{H^1})\\
&\le c(h\|u\|_{C^0(I_n,H^2)} +\tau \|\p_t u\|_{C^0(I_n,H^1)}).
\end{split}\]
This implies the estimate for $s_h^n$. \\
(ii) To bound $q_h^n$ we use Lemma~\ref{lemma:lip_ph} to verify that 
\[
\|q_h^n\|_{H^k} \le c \|d_t u_{*,h}^n\|_{W^{1,\infty}} \|u_{*,h}^{n-1}\|_{W^{1,\infty}} \|u_{*,h}^{n-1}- u_h^{n-1}\|_{H^k}.
\]
The $W^{1,\infty}$ stability of $R_h$ implies that 
$\| u_{*,h}^{n-1}\|_{W^{1,\infty}} \le c\|u\|_{C^0(I_n,W^{1,\infty})}$
and, incorporating~\eqref{stab_diff_quot}, imply
\[
\|d_t u_{*,h}^n\|_{W^{1,\infty}} \le c \|d_t u_*^{n-1}\|_{W^{1,\infty}} 
\le c \|\p_t u\|_{C^0(I_n,W^{1,\infty})}. 
\]
A combination of the estimates proves the estimates for $\|q_h^n\|_{H^k}$. 
\end{proof}

\begin{remark}
In the application of Lemmas~\ref{consistency},~\ref{res_estimate}, 
and~\ref{test_fn_corr} only weighted 
sums of the squares of $\|d_h^n\|$ and $\|s_h^n\|_{H^1}$ are needed, cf.~Lemma~\ref{stability}. 
Therefore, the conditions $u\in C^1([0,T],H^1(\O)\cap W^{1,\infty}(\O))$
and $u\in C^2([0,T],H^1(\O))$ can be replaced by the weaker requirements
$u\in H^1([0,T],H^1(\O)\cap W^{1,\infty}(\O))$ and $u\in H^2([0,T],H^1(\O))$,
respectively.
\end{remark}

% \cleardoublepage

\section{Error analysis}\label{sec:error_est}
The following lemma provides a discrete error estimate for the difference
between the numerical solutions and the Ritz projections of a regular 
solution. It results from a stability argument for the error equation~\eqref{error}.

\begin{lemma}[Error equation stability]\label{stability}
Let $(u_h^n)_{n=0,\dots,N}$ solve~\eqref{full_discr}, and define $u_{*,h}^n = R_h(u(t_n)))$ for
a solution $u$ of~\eqref{hhf_strong} satisfying~\eqref{eq:reg_cond}.
Then for $h,\tau>0$ sufficiently small, 
the discrete error $e_h^n = u_h^n - u_{*,h}^n$ satisfies 
\begin{equation}\label{stab_ineq}
\max _{n=0,\dots,N}\|e_h^n\|_{H^1}^2 \le c_{{\rm stab}} B_{h,\tau}^2,
\end{equation}
where $B_{h,\tau}$ is defined via
\[
B_{h,\tau}^2 = \|e_h^0\|_{H^1}^2+\tau \sum_{n=1}^N \big(\|d_h^n\|^2 + \|s_h^n\|_{H^1}^2 \big) ,
\]
provided that $B_{h,\tau}^2 \le c_B^2 h$ with $c_B$ sufficiently small. 
\end{lemma}

If $\| e_h^0\|_{H^1}\le ch$ then Lemmas~\ref{consistency} and~\ref{test_fn_corr}
imply that $B_{h,\tau} \le c (h+\tau)$ so that the condition of Lemma~\ref{stability} is
satisfied if $\tau \le c_{{\rm m}} h^{1/2}$ with $c_{\rm m}>0$ sufficiently small. 
The error estimates for the Ritz projections imply that
\[ 
\max_{n=0,\dots,N} \|u(t_n) - u_h^n\|_{H^1} \le c(h+\tau).
\]
This implies the result of Theorem~\ref{thm:error estiamate}. Moreover, 
the inverse estimate~\eqref{eq:inv_est_p} shows $1/2\le |u_h^{n-1}|\le 2$ 
and Lemma~\ref{la:normalize_lip} implies that 
for the normalized approximations $N(u_h^n)$ we have
\[
\|u(t_n) - N(u_h^n)\|_{H^1} = \|N(u(t_n)) - N(u_h^n)\|_{H^1} \le c \|u(t_n) - u_h^n\|_{H^1},
\]
so that these satisfy the same approximation properties. We note that
in view of a sharper inverse estimate the step-size condition can be
weakened to $\tau\le c_{{\rm m}} (1+|\log h|)^{-1}$ if $n_\O = 2$ and
$\tau \le c_{{\rm m}}$ if $n_\O=1$.

\begin{proof}
(i) To ensure the stability of the normalization the uniform bound 
$1/2 \le |u_h^{n-1}| \le 2$ is needed. We argue by induction and assume that~\eqref{stab_ineq}  
holds with $N$ replaced by $N'-1$. For $N'-1=0$ this is 
satisfied since by definition of $B_{h,\tau}$ we have $\|e_h^0\|_{H^1} \le B_{h,\tau}$. 
Then, the inverse estimate~\eqref{eq:inv_est_p} 
and the assumption on $B_{h,\tau}$ with $c_B$ small enough imply that
\[ %begin{equation}\label{ind}
\|e_h^{n-1}\|_{L^\infty}^2 \le c_{{\rm inv}}^2 h^{-1}\|e_h^{n-1}\|_{H^1}^2 
\le c _{{\rm inv}}^2h^{-1} c_{{\rm stab}} B_{h,\tau}^2 \le c_{{\rm inv}}^2 c_{{\rm stab}} c_B^2 \le \frac1{16},
\]
for all $n\le N'$. We then deduce with~\eqref{ritz_max_norm} that 
\[		
\| |u_h^{n-1}| -1 \|_{L^\infty} \le  \| u_h^{n-1} - u_*^{n-1}\|_{L^\infty} 
\le \| e_h^{n-1} \|_{L^\infty} + \|u_{*,h}^{n-1} - u_*^{n-1}\|_{L^\infty} \le \frac12
\]
for $h$ sufficiently small. \\
(ii) For $n\le N'$ we test the error equation~\eqref{error} by 
$\phi_h = P_h(\hu_h^n) d_t e_h^n \in T_h (\hu_h^n)$
and use Lemma~\ref{test_fn_corr}, which shows $\phi_h = d_t e_h^n - s_h^n - q_h^n$. 
This leads to 
\[\begin{split}
\|d_t e_h^n\|^2 + (\nabla e_h^n,\nabla d_t e_h^n) 
& = (d_t e_h^n, s_h^n + q_h^n) + (\nabla e_h^n, \nabla (s_h^n+q_h^n)) \\ 
& \quad - (r_h^n,d_t e_h^n - s_h^n - q_h^n).
\end{split}\]
With a binomial formula and H\"older and Young inequalites we deduce that
\begin{equation}\label{eq:err_1} 
\begin{split} 
& \|d_t e_h^n\|^2 + \frac{d_t}{2} \|\nabla e_h^n\|^2 + \frac{\tau}{2} \|\nabla d_t e_h^n\|^2 \\
& \le \frac12 \|d_t e_h^n\|^2 + 2 \|s_h^n + q_h^n\|^2 + 2 \|r_h^n\|^2 + \frac12 \|\nabla e_h^n\|^2
+ \frac12 \|\nabla (s_h^n + q_h^n)\|^2. 
\end{split}
\end{equation}
To obtain the full $H^1$-norm of $e_h^n$ on the left-hand side we note that
\[ 
\frac{d_t}2 \|e_h^n\|^2 + \frac\tau2 \|d_t e_h^n\|^2 = (e_h^n,d_t e_h^n)  
\le \frac12 \|e_h^n\|^2 + \frac12 \|d_t e_h^n\|^2.
\] 
Hence, by adding $(1/2) \|e_h^n\|^2$ to both sides of~\eqref{eq:err_1} we find
\[ % begin{equation}\label{eq:err_2} 
 \frac{d_t}{2} \|e_h^n\|_{H^1}^2  
 \le  2 \|s_h^n + q_h^n\|^2 + 2 \|r_h^n\|^2 + \frac12 \| e_h^n\|_{H^1}^2
+ \frac12 \|\nabla (s_h^n + q_h^n)\|^2. 
\]
Multiplication by $2 \tau$ and summation over $n= 1,2,\dots,n'$ with $n'\le N'$
show that 
\[
\|e_h^{n'}\|_{H^1}^2 \le \|e_h^0\|_{H^1}^2 + \tau \sum_{n=1}^{n'} \| e_h^n\|_{H^1}^2 
+ 4 \tau \sum_{n=1}^{n'} \big(\|s_h^n + q_h^n\|_{H^1}^2 + \|r_h^n\|^2 \big). 
\]
Absorbing $\|e_h^{n'}\|_{H^1}^2$ for $\tau$ sufficiently small
and incorporating Lemmas~\ref{res_estimate} and~\ref{test_fn_corr} leads to 
\[ 
\|e_h^{n'}\|_{H^1}^2 \le c_1 \tau \sum_{n=1}^{n'-1}\|e_h^n\|_{H^1}^2  + c_2 B_{h,\tau}^2
\]
for all $n'=1,2,\dots,N'$. A discrete Gronwall inequality proves 
\[
\max_{n'=1,...,N'}\|e_h^{n'}\|_{H^1}^2 \le c_{{\rm stab}} B_{h,\tau}^2,
\]
with $c_{{\rm stab}} = \exp(c_1 \tau N) \le \exp(c_1 T)$ independently of $N'$. 
Hence,~\eqref{stab_ineq} holds with $N'$ and this completes the induction argument. 
\end{proof}

% \cleardoublepage

\medskip
\noindent
{\em Acknowledgments.}
The authors acknowledge support by the German Research Foundation (DFG) 
via research unit FOR 3013 {\em Vector- and tensor-valued surface PDEs}
(Grant no. BA2268/6–1). The work of Bal\'azs Kov\'acs is funded by the 
Heisenberg Programme of the DFG (Project-ID 446431602).

\section*{References}
\printbibliography[heading=none]

@book {Bart15-book,
    AUTHOR = {Bartels, S\"{o}ren},
     TITLE = {Numerical methods for nonlinear partial differential
              equations},
    SERIES = {Springer Series in Computational Mathematics},
    VOLUME = {47},
 PUBLISHER = {Springer, Cham},
      YEAR = {2015},
     PAGES = {x+393},
      ISBN = {978-3-319-13796-4; 978-3-319-13797-1},
   MRCLASS = {65-01 (35A15 35A35 65Mxx 65Nxx)},
  MRNUMBER = {3309171},
MRREVIEWER = {Karsten Urban},
       DOI = {10.1007/978-3-319-13797-1},
       URL = {https://doi.org/10.1007/978-3-319-13797-1},
}

@book {BreSco08-book,
    AUTHOR = {Brenner, Susanne C. and Scott, L. Ridgway},
     TITLE = {The mathematical theory of finite element methods},
    SERIES = {Texts in Applied Mathematics},
    VOLUME = {15},
   EDITION = {Third},
 PUBLISHER = {Springer, New York},
      YEAR = {2008},
     PAGES = {xviii+397},
      ISBN = {978-0-387-75933-3},
   MRCLASS = {65-01 (65-02)},
  MRNUMBER = {2373954},
       DOI = {10.1007/978-0-387-75934-0},
       URL = {https://doi.org/10.1007/978-0-387-75934-0},
}

@article {Alouges2008,
	AUTHOR = {Alouges, Fran\c{c}ois},
	TITLE = {A new finite element scheme for {L}andau-{L}ifchitz equations},
	JOURNAL = {Discrete Contin. Dyn. Syst. Ser. S},
	FJOURNAL = {Discrete and Continuous Dynamical Systems. Series S},
	VOLUME = {1},
	YEAR = {2008},
	NUMBER = {2},
	PAGES = {187--196},
	ISSN = {1937-1632},
	MRCLASS = {65M60 (35K55 82B80 82D40)},
	MRNUMBER = {2379897},
	MRREVIEWER = {Etienne Emmrich},
	DOI = {10.3934/dcdss.2008.1.187},
	URL = {https://doi.org/10.3934/dcdss.2008.1.187},
}

@article {AlougesJaisson2006,
	AUTHOR = {Alouges, Fran\c{c}ois and Jaisson, Pascal},
	TITLE = {Convergence of a finite element discretization for the
	{L}andau-{L}ifshitz equations in micromagnetism},
	JOURNAL = {Math. Models Methods Appl. Sci.},
	FJOURNAL = {Mathematical Models and Methods in Applied Sciences},
	VOLUME = {16},
	YEAR = {2006},
	NUMBER = {2},
	PAGES = {299--316},
	ISSN = {0218-2025},
	MRCLASS = {65M60 (35K55 35Q60 82D40)},
	MRNUMBER = {2210092},
	MRREVIEWER = {B\'{e}atrice M. Rivi\`ere},
	DOI = {10.1142/S0218202506001169},
	URL = {https://doi.org/10.1142/S0218202506001169},
}

@article {AKST2014,
	AUTHOR = {Alouges, Fran\c{c}ois and Kritsikis, Evaggelos and Steiner, Jutta
	and Toussaint, Jean-Christophe},
	TITLE = {A convergent and precise finite element scheme for
	{L}andau-{L}ifschitz-{G}ilbert equation},
	JOURNAL = {Numer. Math.},
	FJOURNAL = {Numerische Mathematik},
	VOLUME = {128},
	YEAR = {2014},
	NUMBER = {3},
	PAGES = {407--430},
	ISSN = {0029-599X},
	MRCLASS = {65M60 (35K55 35Q60 65M12)},
	MRNUMBER = {3268842},
	MRREVIEWER = {Nicolae Pop},
	DOI = {10.1007/s00211-014-0615-3},
	URL = {https://doi.org/10.1007/s00211-014-0615-3},
}

@article {GuiLiWang2022,
	AUTHOR = {Gui, Xinping and Li, Buyang and Wang, Jilu},
	TITLE = {Convergence of renormalized finite element methods for heat
	flow of harmonic maps},
	JOURNAL = {SIAM J. Numer. Anal.},
	FJOURNAL = {SIAM Journal on Numerical Analysis},
	VOLUME = {60},
	YEAR = {2022},
	NUMBER = {1},
	PAGES = {312--338},
	ISSN = {0036-1429},
	MRCLASS = {65M60 (35K55 35Q35 65M12)},
	MRNUMBER = {4377027},
	DOI = {10.1137/21M1402212},
	URL = {https://doi.org/10.1137/21M1402212},
}

@article {AFKL2021,
	AUTHOR = {Akrivis, Georgios and Feischl, Michael and Kov\'{a}cs, Bal\'{a}zs and
	Lubich, Christian},
	TITLE = {Higher-order linearly implicit full discretization of the
	{L}andau-{L}ifshitz-{G}ilbert equation},
	JOURNAL = {Math. Comp.},
	FJOURNAL = {Mathematics of Computation},
	VOLUME = {90},
	YEAR = {2021},
	NUMBER = {329},
	PAGES = {995--1038},
	ISSN = {0025-5718},
	MRCLASS = {65M60 (35Q60 65L06 65M12 65M15)},
	MRNUMBER = {4232216},
	MRREVIEWER = {Hamdullah Y\"{u}cel},
	DOI = {10.1090/mcom/3597},
	URL = {https://doi.org/10.1090/mcom/3597},
}

@article {An2016,
	AUTHOR = {An, Rong},
	TITLE = {Optimal error estimates of linearized {C}rank-{N}icolson
	{G}alerkin method for {L}andau-{L}ifshitz equation},
	JOURNAL = {J. Sci. Comput.},
	FJOURNAL = {Journal of Scientific Computing},
	VOLUME = {69},
	YEAR = {2016},
	NUMBER = {1},
	PAGES = {1--27},
	ISSN = {0885-7474},
	MRCLASS = {65M60 (35K55 35Q60 65M15)},
	MRNUMBER = {3542789},
	MRREVIEWER = {Dmitriy Leykekhman},
	DOI = {10.1007/s10915-016-0181-1},
	URL = {https://doi.org/10.1007/s10915-016-0181-1},
}

@article {BanasProhlSchaetzle2010,
	AUTHOR = {Ba\v{n}as, \v{L}ubom\'{\i}r and Prohl, Andreas and Sch\"{a}tzle, Reiner},
	TITLE = {Finite element approximations of harmonic map heat flows and
	wave maps into spheres of nonconstant radii},
	JOURNAL = {Numer. Math.},
	FJOURNAL = {Numerische Mathematik},
	VOLUME = {115},
	YEAR = {2010},
	NUMBER = {3},
	PAGES = {395--432},
	ISSN = {0029-599X},
	MRCLASS = {35A35 (35A01 35D30 35K60 35L20 65M60)},
	MRNUMBER = {2640052},
	DOI = {10.1007/s00211-009-0282-y},
	URL = {https://doi.org/10.1007/s00211-009-0282-y},
}

@article {BartelsLubichProhl2009,
	AUTHOR = {Bartels, S\"{o}ren and Lubich, Christian and Prohl, Andreas},
	TITLE = {Convergent discretization of heat and wave map flows to
	spheres using approximate discrete {L}agrange multipliers},
	JOURNAL = {Math. Comp.},
	FJOURNAL = {Mathematics of Computation},
	VOLUME = {78},
	YEAR = {2009},
	NUMBER = {267},
	PAGES = {1269--1292},
	ISSN = {0025-5718},
	MRCLASS = {65M12 (35B40 35K20 35K55 35L70 65M60)},
	MRNUMBER = {2501050},
	MRREVIEWER = {Veronika Sobot\'{\i}kov\'{a}},
	DOI = {10.1090/S0025-5718-09-02221-2},
	URL = {https://doi.org/10.1090/S0025-5718-09-02221-2},
}

@article {BartelsProhl2006,
	AUTHOR = {Bartels, S\"{o}ren and Prohl, Andreas},
	TITLE = {Convergence of an implicit finite element method for the
	{L}andau-{L}ifshitz-{G}ilbert equation},
	JOURNAL = {SIAM J. Numer. Anal.},
	FJOURNAL = {SIAM Journal on Numerical Analysis},
	VOLUME = {44},
	YEAR = {2006},
	NUMBER = {4},
	PAGES = {1405--1419},
	ISSN = {0036-1429},
	MRCLASS = {65M60 (82D40)},
	MRNUMBER = {2257110},
	MRREVIEWER = {Anne Nouri},
	DOI = {10.1137/050631070},
	URL = {https://doi.org/10.1137/050631070},
}

@article {BartelsProhl2007,
	AUTHOR = {Bartels, S\"{o}ren and Prohl, Andreas},
	TITLE = {Constraint preserving implicit finite element discretization
	of harmonic map flow into spheres},
	JOURNAL = {Math. Comp.},
	FJOURNAL = {Mathematics of Computation},
	VOLUME = {76},
	YEAR = {2007},
	NUMBER = {260},
	PAGES = {1847--1859},
	ISSN = {0025-5718},
	MRCLASS = {65M60 (35K55 53C44 58E20)},
	MRNUMBER = {2336271},
	MRREVIEWER = {Beny Neta},
	DOI = {10.1090/S0025-5718-07-02026-1},
	URL = {https://doi.org/10.1090/S0025-5718-07-02026-1},
}

@article {Cimrak2005,
	AUTHOR = {Cimr\'{a}k, Ivan},
	TITLE = {Error estimates for a semi-implicit numerical scheme solving
	the {L}andau-{L}ifshitz equation with an exchange field},
	JOURNAL = {IMA J. Numer. Anal.},
	FJOURNAL = {IMA Journal of Numerical Analysis},
	VOLUME = {25},
	YEAR = {2005},
	NUMBER = {3},
	PAGES = {611--634},
	ISSN = {0272-4979},
	MRCLASS = {82D40 (78M10 82-04)},
	MRNUMBER = {2153750},
	DOI = {10.1093/imanum/dri011},
	URL = {https://doi.org/10.1093/imanum/dri011},
}

@article {Gao2014,
	AUTHOR = {Gao, Huadong},
	TITLE = {Optimal error estimates of a linearized backward {E}uler {FEM}
	for the {L}andau-{L}ifshitz equation},
	JOURNAL = {SIAM J. Numer. Anal.},
	FJOURNAL = {SIAM Journal on Numerical Analysis},
	VOLUME = {52},
	YEAR = {2014},
	NUMBER = {5},
	PAGES = {2574--2593},
	ISSN = {0036-1429},
	MRCLASS = {65M60 (35K51 35Q60 65M12 78M12)},
	MRNUMBER = {3273326},
	MRREVIEWER = {Sarangam Majumdar},
	DOI = {10.1137/130936476},
	URL = {https://doi.org/10.1137/130936476},
}

@article {AnGaoSun2021,
	AUTHOR = {An, Rong and Gao, Huadong and Sun, Weiwei},
	TITLE = {Optimal error analysis of {E}uler and {C}rank-{N}icolson
	projection finite difference schemes for {L}andau-{L}ifshitz
	equation},
	JOURNAL = {SIAM J. Numer. Anal.},
	FJOURNAL = {SIAM Journal on Numerical Analysis},
	VOLUME = {59},
	YEAR = {2021},
	NUMBER = {3},
	PAGES = {1639--1662},
	ISSN = {0036-1429},
	MRCLASS = {65M06 (35K61 65M12 65M15)},
	MRNUMBER = {4272916},
	DOI = {10.1137/20M1335431},
	URL = {https://doi.org/10.1137/20M1335431},
}

@article{AnSun2021,
	author = {An, Rong and Sun, Weiwei},
	title = "{Analysis of backward Euler projection FEM for the Landau--Lifshitz equation}",
	journal = {IMA Journal of Numerical Analysis},
	year = {2021},
	issn = {0272-4979},
	doi = {10.1093/imanum/drab038},
	url = {https://doi.org/10.1093/imanum/drab038},
	note = {DOI:10.1093/imanum/drab038},
}

@book {Proh01-book,
    AUTHOR = {Prohl, Andreas},
     TITLE = {Computational micromagnetism},
    SERIES = {Advances in Numerical Mathematics},
 PUBLISHER = {B. G. Teubner, Stuttgart},
      YEAR = {2001},
     PAGES = {xviii+304},
      ISBN = {3-519-00358-9},
   MRCLASS = {82D40 (65N30 74F15 82-02 82C80 82D30)},
  MRNUMBER = {1885923},
MRREVIEWER = {Thomas P. Svobodny},
       DOI = {10.1007/978-3-663-09498-2},
       URL = {https://doi.org/10.1007/978-3-663-09498-2},
}

@article {MPPR22,
    AUTHOR = {Mauser, Norbert J. and Pfeiler, Carl-Martin and Praetorius,
              Dirk and Ruggeri, Michele},
     TITLE = {Unconditional well-posedness and {IMEX} improvement of a
              family of predictor-corrector methods in micromagnetics},
   JOURNAL = {Appl. Numer. Math.},
  FJOURNAL = {Applied Numerical Mathematics. An IMACS Journal},
    VOLUME = {180},
      YEAR = {2022},
     PAGES = {33--54},
      ISSN = {0168-9274},
   MRCLASS = {65M60 (65M12 78M10)},
  MRNUMBER = {4432060},
       DOI = {10.1016/j.apnum.2022.05.008},
       URL = {https://doi.org/10.1016/j.apnum.2022.05.008},
}

@article {PRS18,
    AUTHOR = {Praetorius, Dirk and Ruggeri, Michele and Stiftner, Bernhard},
     TITLE = {Convergence of an implicit-explicit midpoint scheme for
              computational micromagnetics},
   JOURNAL = {Comput. Math. Appl.},
  FJOURNAL = {Computers \& Mathematics with Applications. An International
              Journal},
    VOLUME = {75},
      YEAR = {2018},
    NUMBER = {5},
     PAGES = {1719--1738},
      ISSN = {0898-1221},
   MRCLASS = {65M60 (82D40)},
  MRNUMBER = {3766546},
       DOI = {10.1016/j.camwa.2017.11.028},
       URL = {https://doi.org/10.1016/j.camwa.2017.11.028},
}

@article {EellSamp64,
    AUTHOR = {Eells, Jr., James and Sampson, J. H.},
     TITLE = {Harmonic mappings of {R}iemannian manifolds},
   JOURNAL = {Amer. J. Math.},
  FJOURNAL = {American Journal of Mathematics},
    VOLUME = {86},
      YEAR = {1964},
     PAGES = {109--160},
      ISSN = {0002-9327},
   MRCLASS = {53.72 (57.50)},
  MRNUMBER = {164306},
MRREVIEWER = {J. A. Wolf},
       DOI = {10.2307/2373037},
       URL = {https://doi.org/10.2307/2373037},
}

@article {ChenStru89,
    AUTHOR = {Chen, Yun Mei and Struwe, Michael},
     TITLE = {Existence and partial regularity results for the heat flow for
              harmonic maps},
   JOURNAL = {Math. Z.},
  FJOURNAL = {Mathematische Zeitschrift},
    VOLUME = {201},
      YEAR = {1989},
    NUMBER = {1},
     PAGES = {83--103},
      ISSN = {0025-5874},
   MRCLASS = {58E20 (58G11)},
  MRNUMBER = {990191},
MRREVIEWER = {Helmut Kaul},
       DOI = {10.1007/BF01161997},
       URL = {https://doi.org/10.1007/BF01161997},
}

@article {Bart05,
    AUTHOR = {Bartels, S\"{o}ren},
     TITLE = {Stability and convergence of finite-element approximation
              schemes for harmonic maps},
   JOURNAL = {SIAM J. Numer. Anal.},
  FJOURNAL = {SIAM Journal on Numerical Analysis},
    VOLUME = {43},
      YEAR = {2005},
    NUMBER = {1},
     PAGES = {220--238},
      ISSN = {0036-1429},
   MRCLASS = {65N30 (58E20 76A15 82D30)},
  MRNUMBER = {2177142},
MRREVIEWER = {R. Kodn\'{a}r},
       DOI = {10.1137/040606594},
       URL = {https://doi.org/10.1137/040606594},
}

@article {Alou97,
    AUTHOR = {Alouges, Fran\c{c}ois},
     TITLE = {A new algorithm for computing liquid crystal stable
              configurations: the harmonic mapping case},
   JOURNAL = {SIAM J. Numer. Anal.},
  FJOURNAL = {SIAM Journal on Numerical Analysis},
    VOLUME = {34},
      YEAR = {1997},
    NUMBER = {5},
     PAGES = {1708--1726},
      ISSN = {0036-1429},
   MRCLASS = {82D30 (65C20 76A15 76M25)},
  MRNUMBER = {1472192},
MRREVIEWER = {Denis Serre},
       DOI = {10.1137/S0036142994264249},
       URL = {https://doi.org/10.1137/S0036142994264249},
}

@book {Stru96-book,
    AUTHOR = {Struwe, Michael},
     TITLE = {Variational methods},
    VOLUME = {34},
   EDITION = {Second},
 PUBLISHER = {Springer-Verlag, Berlin},
      YEAR = {1996},
     PAGES = {xvi+272},
      ISBN = {3-540-58859-0},
   MRCLASS = {49-02 (34C25 35A15 35F20 47H15 58E30)},
  MRNUMBER = {1411681},
MRREVIEWER = {Bernhard Kawohl},
       DOI = {10.1007/978-3-662-03212-1},
       URL = {https://doi.org/10.1007/978-3-662-03212-1},
}

@article {ChDiYe92,
    AUTHOR = {Chang, Kung-Ching and Ding, Wei Yue and Ye, Rugang},
     TITLE = {Finite-time blow-up of the heat flow of harmonic maps from
              surfaces},
   JOURNAL = {J. Differential Geom.},
  FJOURNAL = {Journal of Differential Geometry},
    VOLUME = {36},
      YEAR = {1992},
    NUMBER = {2},
     PAGES = {507--515},
      ISSN = {0022-040X},
   MRCLASS = {58E20 (35K55 58G11)},
  MRNUMBER = {1180392},
MRREVIEWER = {Martin Fuchs},
       URL = {http://projecteuclid.org/euclid.jdg/1214448751},
}

@article {Bart16,
    AUTHOR = {Bartels, S\"{o}ren},
     TITLE = {Projection-free approximation of geometrically constrained
              partial differential equations},
   JOURNAL = {Math. Comp.},
  FJOURNAL = {Mathematics of Computation},
    VOLUME = {85},
      YEAR = {2016},
    NUMBER = {299},
     PAGES = {1033--1049},
      ISSN = {0025-5718},
   MRCLASS = {65J05 (65M12 65M60 65N12 65N30)},
  MRNUMBER = {3454357},
MRREVIEWER = {Waleed M. Abd-Elhameed},
       DOI = {10.1090/mcom/3008},
       URL = {https://doi.org/10.1090/mcom/3008},
}

@article {KPPRS19,
    AUTHOR = {Kraus, Johannes and Pfeiler, Carl-Martin and Praetorius, Dirk
              and Ruggeri, Michele and Stiftner, Bernhard},
     TITLE = {Iterative solution and preconditioning for the tangent plane
              scheme in computational micromagnetics},
   JOURNAL = {J. Comput. Phys.},
  FJOURNAL = {Journal of Computational Physics},
    VOLUME = {398},
      YEAR = {2019},
     PAGES = {108866, 27},
      ISSN = {0021-9991},
   MRCLASS = {65M99 (65Z05)},
  MRNUMBER = {3995784},
MRREVIEWER = {Khaled Mohammed Saad},
       DOI = {10.1016/j.jcp.2019.108866},
       URL = {https://doi.org/10.1016/j.jcp.2019.108866},
}

@article {BarRei21,
    AUTHOR = {Bartels, S\"{o}ren and Reiter, Philipp},
     TITLE = {Stability of a simple scheme for the approximation of elastic
              knots and self-avoiding inextensible curves},
   JOURNAL = {Math. Comp.},
  FJOURNAL = {Mathematics of Computation},
    VOLUME = {90},
      YEAR = {2021},
    NUMBER = {330},
     PAGES = {1499--1526},
      ISSN = {0025-5718},
   MRCLASS = {65N30 (53A04 65N12)},
  MRNUMBER = {4273107},
       DOI = {10.1090/mcom/3633},
       URL = {https://doi.org/10.1090/mcom/3633},
}

@article {BBFP07,
    AUTHOR = {Barrett, John W. and Bartels, S\"{o}ren and Feng, Xiaobing and
              Prohl, Andreas},
     TITLE = {A convergent and constraint-preserving finite element method
              for the {$p$}-harmonic flow into spheres},
   JOURNAL = {SIAM J. Numer. Anal.},
  FJOURNAL = {SIAM Journal on Numerical Analysis},
    VOLUME = {45},
      YEAR = {2007},
    NUMBER = {3},
     PAGES = {905--927},
      ISSN = {0036-1429},
   MRCLASS = {65M60 (49J10 58E20 94A08)},
  MRNUMBER = {2318794},
MRREVIEWER = {Erwin Schechter},
       DOI = {10.1137/050639429},
       URL = {https://doi.org/10.1137/050639429},
}

\end{document}